\documentclass[12pt]{amsart}
\usepackage[utf8]{inputenc}
\usepackage{amsmath}
\usepackage{graphics}
\usepackage{amsthm, multicol}
\usepackage{bm}
\usepackage[at]{easylist}
\usepackage{amssymb}
\usepackage{setspace, tikz, float,graphicx, tikz-cd}
\usepackage{mathrsfs}
\usepackage{mathabx}
\usepackage{interval}
\usepackage{amsfonts,latexsym,amscd,euscript,graphicx}
\providecommand{\keywords}[1]{\textbf{\textit{Index terms---}} #1}
\usepackage{framed, mathtools}
\usepackage{fullpage}
\usepackage{hyperref}
\usepackage[margin=0.8 in]{geometry}
\usepackage{setspace}
\usepackage{enumerate}
\newtheorem{thm}{Theorem}
\newtheorem{lemma}[thm]{Lemma}

   \hypersetup{colorlinks=true,citecolor=blue,urlcolor =black,linkbordercolor={1 0 0}}

\allowdisplaybreaks[1]
\usepackage{thmtools}
\declaretheoremstyle[notefont=\bfseries,notebraces={}{},%
    headpunct={},postheadspace=1em,spaceabove=0.5em,spacebelow=0.5em]{mystyle}
\declaretheorem[style=mystyle,numbered=no,name=Theorem]{thm-hand}
\declaretheorem[style=mystyle,numbered=no,name=Conjecture]{conj-hand}
\declaretheorem[style=mystyle,numbered=no,name=Definition.]{defn-hand}
\declaretheorem[style=mystyle,numbered=no,name=Theorem.]{thm-no}
\declaretheorem[style=mystyle,numbered=no,name=Conjecture.]{conj-no}
\declaretheorem[style=mystyle,numbered=no,name=Lemma.]{lemma-no}
\declaretheorem[style=mystyle,numbered=no,name=Lemma]{lemma-hand}
\makeatletter
\def\resetMathstrut@{%
  \setbox\z@\hbox{%
    \mathchardef\@tempa\mathcode`\[\relax
    \def\@tempb##1"##2##3{\the\textfont"##3\char"}%
    \expandafter\@tempb\meaning\@tempa \relax
  }%
  \ht\Mathstrutbox@\ht\z@ \dp\Mathstrutbox@\dp\z@}
\makeatother
\begingroup
  \catcode`(\active \xdef({\left\string(}
  \catcode`)\active \xdef){\right\string)}
  \endgroup
\mathcode`(="8000 \mathcode`)="8000

\newtheorem{proposition}[thm]{Proposition}
\newtheorem{conj}[thm]{Conjecture}
\newtheorem{prob}[thm]{Question}

\theoremstyle{definition}
\newtheorem{defn}[thm]{Definition}

\usepackage{listings}
\theoremstyle{remark}

\newcommand{\Z}{\mathbb Z}

\newcommand{\nc}{\newcommand}

\nc{\on}{\operatorname}
\nc{\Spec}{\on{Spec}}
\vspace{-15mm}
\title{On a Conjecture Regarding Permutations which Destroy Arithmetic Progressions} 

\author{Mehtaab Sawhney}
\thanks{Massachusetts Institute of Technology, Cambridge MA. Email: \texttt{msawhney@mit.edu}}
\author{David Stoner}
\thanks{Harvard University, Cambridge MA. Email: \texttt{dstoner@college.harvard.edu}}
\date{\today}

\begin{document}

\maketitle
\vspace{-10mm}
\begin{abstract}
Hegarty conjectured for $n\neq 2, 3, 5, 7$ that $\Z/n\Z$ has a permutation which destroys all arithmetic progressions mod $n$. For $n\ge n_0$, Hegarty and Martinsson demonstrated that $\mathbb{Z}/n\Z$ has an arithmetic-progression destroying permutation. However $n_0\approx 1.4\times 10^{14}$ and thus resolving the conjecture in full remained out of reach of any computational techniques. However this paper, using constructions modeled after those used by Elkies and Swaminathan for the case of $\Z/p\Z$ with $p$ being prime, establishes the conjecture in full. Furthermore our results do not rely on the fact that it suffices to study when $n<n_0$ and thus our results completely independent of the proof given by Hegarty and Martinsson.
\end{abstract}

\section{Background}
In 2004 Hegarty \cite{hegarty2004permutations} introduced the notion of permutations that destroy arithmetic progressions in finite cyclic groups.
\begin{defn}
Given a permutation $\pi:\Z/n\Z\to \Z/n\Z$, a three term arithmetic progression $(a, a+r, a+2r)$, with not all terms equal, is called \textit{preserved} in $\pi$ if $\pi(a+2r)-2\pi(a+r)+\pi(a)=0$. A permutation is said to \textit{destroy} all arithmetic progressions if it has no preserved arithmetic progressions.
\end{defn}
For the sake of simplicity, a three-term arithmetic progression will be denoted an AP and a permutation that destroys all APs will be called AP-Destroying. This notion can be extended to permutations which destroy $k$-term arithmetic progressions and Hegarty \cite{hegarty2004permutations} demonstrated that for $n\neq 3,4$ there exists a permutation of $\Z/n\Z$ that destroys all $k$-term arithmetic progressions for all $k\ge 4$. However, classifying which cyclic groups have an AP-Destroying permutation has been resistant to proof. In particular Hegarty \cite{hegarty2004permutations} gave the following conjecture regarding AP-Destroying permutations based on computational evidence.
\begin{conj}
For $n\not\in\left\{2,3,5,7\right\}$, there exists an AP-Destroying permutation $\pi:\Z/n\Z\to \Z/n\Z$.
\end{conj}
This conjecture was proved for sufficiently large $n$ by Hegarty and Martinsson \cite{hegarty2015permutations} in 2015.
\begin{thm}\label{current}
For $n\ge(9\times11\times16\times17\times19\times23)^2$, there exists a AP-Destroying permutation $\pi:\Z/n\Z\to \Z/n\Z$.
\end{thm}
However given that $(9\times11\times16\times17\times19\times23)^2\approx 1.4\times 10^{14}$, any purely computational approach is out of reach in order to establish Hegarty's original conjecture. We instead base our construction on that of Elkies and Swaminathan \cite{Ashvin}, who proved the following result.
\begin{thm}
Let $p$ be a prime with $p\ge 11$. Then there exists an AP-Destroying permutation $\pi:\Z/p\Z\to \Z/p\Z$.
\end{thm}
Following this approach we establish the original conjecture of Hegarty \cite{hegarty2004permutations}.
\begin{thm}\label{main}
For $n\not\in\left\{2,3,5,7\right\}$, there exists a AP-Destroying permutation $\pi:\Z/n\Z\to \Z/n\Z$.
\end{thm}
Note that Hegarty \cite{hegarty2004permutations} computationally checked that each of the values $n\in\left\{2,3,5,7\right\}$ does not have an AP-Destroying permutation, so it suffices to prove that the remaining values do have an AP-Destroying permutation.
\section{Preliminary Reductions}
The starting point for our proof is a theorem from Hegarty \cite{hegarty2004permutations} that can be used to simplify the general case to five infinite classes of integers and a finite exceptional set. (Note that the theorem given by Hegarty \cite{hegarty2004permutations} applies more generally for abelian groups.)
\begin{thm}\label{Product}
If there exists an AP-Destroying permutation for $\Z/m\Z$ and $\Z/n\Z$, there exists a AP-Destroying permutation for $\Z/mn\Z$. Note that $m$ and $n$ are not necessarily coprime.
\end{thm}
Given this theorem it is possible to reduce the set of integers necessary to prove the desired result. This reduction is given without proof in Hegarty \cite{hegarty2004permutations}. (There appears to be a slight error in the version given by Hegarty \cite{hegarty2004permutations} as it excludes the case when $n=343$.)
\begin{thm}
In order to prove Theorem \ref{main}, it suffices to prove the cases $\{p, 2p, 3p, 5p, 7p~|~p$ prime and $p\ge 11\}$ and the integers $\left\{p_1p_2,p_1p_2p_3\right\}$ with $p_i\in\left\{2,3,5,7\right\}$, not necessarily distinct.
\end{thm}
\begin{proof}
Suppose that $n=2^{a_1}3^{a_2}5^{a_3}7^{a_4}p_1^{b_1}\ldots p_k^{b_k}$. If $2^{a_1}3^{a_2}5^{a_3}7^{a_4}\not\in\left\{1, 2,3,5,7\right\}$ then find a AP-Destroying permutation for each $p_i^{b_i}$ and $2^{a_1}3^{a_2}5^{a_3}7^{a_4}$ and the result follows from the previous lemma. The last integer can be constructed as $a_1+a_2+a_3+a_4\ge 2$ so we can represent $a_1+a_2+a_3+a_4$ as a sum of $2$'s and $3$'s and using this we can construct $2^{a_1}3^{a_2}5^{a_3}7^{a_4}$ as a product of products of $2$ or $3$ primes in $\left\{2,3,5,7\right\}$. Otherwise we take the $2^{a_1}3^{a_2}5^{a_3}7^{a_4}p_1$ and $\frac{n}{2^{a_1}3^{a_2}5^{a_3}7^{a_4}p_1}$ in order to represent $n$ an find a permutation for each of the integers independently. 
\end{proof}
To prove the result for the cases $\{2p, 3p, 5p, 7p~|~p$ prime and $p\ge 11\}$ we model our construction based on the one used by Elkies and Swaminathan \cite{Ashvin} to demonstrate Theorem 4. The key similarity is the following lemma of Elkies and Swaminathan from \cite{Ashvin}, which we will rely heavily on as well. Note that in the statement below, and elsewhere in this paper we will not distinguish between an arithmetic progression and its reverse.
\begin{lemma}\label{hi}
Suppose that $\pi: \Z/p\Z\to \Z/p\Z$ is the following permutation with $t\neq 0$:
\[\pi:=\begin{cases}
t & x=0
\\ 0 & x=1
\\ \frac{t}{x} & x\notin{0, 1}
\end{cases}.\]
Then the only APs preserved by $\pi$ are $(0, \frac{3}{2}, 3)$, $(\frac{1}{3}, \frac{2}{3}, 1)$. 
\end{lemma}

Elkies and Swaminathan \cite{Ashvin} then performed two transpositions in order to eliminate these preserved APs and this demonstrated the case when $n$ is prime. In the case $n=2p$ we will ``glue" together two such permutations in a careful manner so that there is exactly one preserved AP, and then using a single transposition we eliminate preserved AP. In the remaining cases however we are able to significantly simplify this approach by directly giving an AP that has no arithmetic progression, avoiding the need for a transposition. In each of these cases however we will not simply be able to show $2p, 3p, 5p, 7p$ for all primes $p$ directly; instead, certain character estimates will show it for $p$ sufficiently large. Thus we show the conjecture to be true for all $n\le2500$ using computational techniques, and this will be a starting point for the analysis in the remaining cases. Note that the $5p$ case, where $p>500$ is assumed, is the limiting case here. All mentioned computational files can be found on the corresponding arXiv submission.

One piece of machinery that is used multiple times in this paper is the Hasse-Weil bound. (Elkies and Swaminathan \cite{Ashvin} similarly require such character estimates, but they can make do with the elementary Hasse bound.) Note that the version we are using is equivalent to counting the number of points on the hyperelliptic curve $y^2=g(x)
$ over a finite field and the bound we are using was proven for curves by Weil in \cite{weil1949numbers}.
\begin{thm}\label{hasse-weil}
Let $\mathbb{F}_p$ be the field with $p$ elements, $p$ being prime,
and let $(\frac{\cdot}{p})$ be the Legendre symbol.
If $f \in \mathbb{F}_p[x]$ is a polynomial of degree $2g+1$ or
$2g+2$ such that $g$ is not a constant times a perfect square in $\mathbb{F}_p[x]$, then
\[
\Bigl|\sum_{y \in \Z/p\Z} (\frac{f(y)}{p})\Bigr| \leq 2g\sqrt{p}+1.
\]
\end{thm}

\section{AP-Destroying Permutations for $\Z/2p\Z$}
Our initial construction is the following permutation $\pi_2: \Z/2\Z\times \Z/p\Z\to \Z/2\Z\times \Z/p\Z$, for a parameter $t\notin\left\{0, 1\right\}$ to be chosen later:
\[\pi_2 :=  \setlength{\arraycolsep}{0pt}
  \renewcommand{\arraystretch}{1.2}
  \left\{\begin{array}{l @{\quad} l}
       (0, 0)\to (1, t) & (1, 0)\to (0, 1)
       \\(0, 1)\to (1, 0) & (1, 1)\to (0,0)
       \\(0, x)\to (0, \frac{1}{x}), x\notin\left\{0, 1\right\} & (1, x)\to (1, \frac{t}{x}), x\notin\left\{0, 1\right\}
  \end{array}\right.\]

\begin{lemma}\label{Lemma2pPartA} 
Suppose that
\[t\not\in \left\{0,1,\frac{1}{4},4,\frac{1}{9},9\right\} \]
and furthermore
\[(\frac{1-\frac{1}{t}}{p})=(\frac{1-t}{p})=-1.\]
Then the only three term arithmetic progressions preserved by $\pi_2$ are $\left\{(0, 1), (1,1), (0, 1)\right\} $ and $\left\{(1, 1), (0, 1), (1, 1)\right\}$.
\end{lemma}
\begin{proof}
We proceed via contradiction. Suppose that $t$ satisfies the above properties, and that some other three term arithmetic progression $T$ is preserved. Let $U$ be the image of $T$. Furthermore denote by $T_2$ and $T_p$ the$\mod 2$ and$\mod p$ components of $T$, respectively, and define $U_2$ and $U_p$ similarly. We separate cases based on the numbers of elements of $T_p$ which are in $\{0,1\}$.
\begin{enumerate}[\text{Case }1.]
\item$T_p$ is of the form $(a-r, a, a+r)$ with $\left\{a-r, a, a+r\right\}\cap \left\{0, 1\right\}=\emptyset$. We take cases which exhaust the possible values of $T_2$.
 \begin{enumerate}[\text{Case 1.}a.]
 \item $T_2=(0, 0, 0)$ or $(1,1,1)$. Then $U_p=(\frac{1}{a-r}, \frac{1}{a}, \frac{1}{a+r})$ or $U_p=(\frac{t}{a-r}, \frac{t}{a}, \frac{t}{a+r})$ depending on $T_2$. In either case, since $t\nequiv 0$, $U_p$ being an AP is equivalent to $\frac{2}{a}\equiv \frac{1}{a-r}+\frac{1}{a+r}\mod p$, which is equivalent to $r^2\equiv 0$. However, this is impossible as $T$ would then be a degenerate AP.
 \item $T_2=(0, 1, 0).$ Then $U_p=(\frac{1}{a-r}, \frac{t}{a}, \frac{1}{a+r})$ and $U_p$ being an AP is equivalent to $\frac{2t}{a}\equiv \frac{1}{a-r}+\frac{1}{a+r}\mod p$. This is equivalent to$(\frac{r}{a})^2\equiv 1-\frac{1}{t}$, which is impossible as $(\frac{1-\frac{1}{t}}{p})=-1$.
\item $T_2=(1, 0, 1).$ Then $U_p=(\frac{t}{a-r}, \frac{1}{a}, \frac{t}{a+r})$. Hence $\frac{2}{a}\equiv \frac{t}{a-r}+\frac{t}{a+r}\mod p$, which is equivalent with $(\frac{r}{a})^2\equiv1-t$. However, this is impossible as $(\frac{1-t}{p})=-1$.
 \end{enumerate}
\item We now consider the case where $|T_p\cap \{0,1\}|=1$. It therefore follows, reversing the AP if necessary, that either $T_p=(1,1+r,1+2r)$, $T_p=(1-r,1,1+r)$, $T_p=(0,r,2r)$, or $T_p=(-r,0,r)$. 
\begin{enumerate}[\text{Case 2.}a.]
\item $T_p=(1,1+r,1+2r).$ Note that if $T_2=(0,w_1,w_2)$ then $U_2=(1,w_1,w_2)$ or vice versa and both of these can not be APs.

\item $T_p=(1-r,1,1+r).$ There are now four possible cases of $T_2$. If $T_2=(0,0,0)$ or $(0,1,0)$ then $U_p=(\frac{1}{1-r},0,\frac{1}{1+r})$. This being an AP is equivalent to $\frac{1}{1-r}+\frac{1}{1+r}\equiv 0$. Simplifying, this is equivalent to $\frac{2}{1-r^2}\equiv 0$ which is impossible. If $T_2=(1,0,1)$ or $(1,1,1)$ then $U_p=(\frac{t}{1-r},0,\frac{t}{1+r})$. This being an AP is equivalent to $\frac{t}{1-r}+\frac{t}{1+r}\equiv 0$. Simplifying, this is equivalent to $\frac{2t}{1-r^2}\equiv 0$ which is impossible as $t\nequiv 0$.
\item $T_p=(0,r,2r).$ There are now four possible cases of $T_2$. If $T_2=(0,0,0)$ then $U_2=(1,0,0)$ which is not an AP modulo $2$. If $T_2=(1,1,1)$ then $U_2=(0,1,1)$ which is not an AP modulo $2$. If $T_2=(0,1,0)$ then $U_2=(1,1,0)$ which is not an AP modulo $2$. Finally if $T_2=(1,0,1)$ then $U_2=(0,0,1)$ which is not an AP modulo $2$. 
\item $T_p=(-r,0,r).$ There are now four possible cases of $T_2$. If $T_2=(0,0,0)$ then $U_p=(\frac{-1}{r},t,\frac{1}{r})$ which is not an AP as $t\nequiv 0$. If $T_2=(0,1,0)$ then $U_p=(\frac{-1}{r},1,\frac{1}{r})$ which is not an AP as $1\nequiv 0$. If $T_2=(1,1,1)$ then $U_p=(\frac{-t}{r},1,\frac{t}{r})$ which is not an AP as $1\nequiv 0$. If $T_2=(1,0,1)$ then $U_p=(\frac{-t}{r},t,\frac{t}{r})$ which is not an AP as $t\nequiv 0$. 
\end{enumerate}
\item In the final case we have that at least two elements of $T_p$ are in $\{0,1\}.$ Reversing the AP if necessary, this gives the cases $T_p=(0,0,0)$, $(1,1,1)$, $(0,\frac{1}{2},1)$, $(0,1,2)$, or $(-1,0,1)$. The second case gives exactly the APs mentioned in the statement of the lemma and therefore it suffices to study the other four cases.
\begin{enumerate}[\text{Case 3.}a.]
\item $T_p=(0,0,0)$. In order for $T$ to not be a trivial progression, $T_2=(0,1,0)$ or $(1,0,1).$ In the first case, $U_p=(t,1,t)$ which is not an AP as $t\nequiv 1$. In the second case, $U_p=(1,t,1)$ which is not an AP as $t\nequiv 1$.
\item $T_p=(0,\frac{1}{2},1).$ If $T_2=(0,0,0)$ then $U_p=(t,2,0)$ but $t\nequiv 4$. If $T_2=(1,1,1)$ then $U_p=(1,2t,0)$ but $t\nequiv \frac{1}{4}$. If $T_2=(0,1,0)$ then $U_p=(t,2t,0)$ but $t\nequiv 0$. Finally if $T_2=(1,0,1)$ then $U_p=(1,2,0)$ which is never an AP.
\item $T_p=(0,1,2).$ If $T_2=(0,0,0)$ then $U_2=(1,1,0)$. If $T_2=(1,1,1)$ then $U_2=(0,0,1)$. If $T_2=(0,1,0)$ then $U_2=(1,0,0)$. Finally if $T_2=(1,0,1)$ then $U_2=(0,1,1)$. In none of these cases is $U_2$ an AP.
\item $T_p=(-1,0,1).$ If $T_2=(0,0,0)$ then $U_2=(0,1,1)$. If $T_2=(1,1,1)$ then $U_2=(1,0,0)$. If $T_2=(0,1,0)$ then $U_2=(0,0,1)$. Finally if $T_2=(1,0,1)$ then $U_2=(1,1,0)$. In none of these cases is $U_2$ an AP.
\end{enumerate}
\end{enumerate}
\end{proof}
Now we claim that a $t$ with the conditions of the previous lemma exists for every prime $p\ge 31$. 
\begin{lemma}\label{Lemma2pPart2} 
For $p\ge 31$, there exists a $t$ such that  
\[t\not\in \left\{0,1,\frac{1}{4},4,\frac{1}{9},9\right\} \]
and
\[(\frac{1-\frac{1}{t}}{p})=(\frac{1-t}{p})=-1.\]
\end{lemma}
\begin{proof}
First note that $(\frac{1-\frac{1}{t}}{p})=(\frac{t(t-1)}{p})$ for $t\nequiv 0$. Then note that
\begin{align*}
\sum_{t\in \Z/p\Z}(1-(\frac{1-t}{p}))(1-(\frac{t(t-1)}{p}))
&= \sum_{t\in \Z/p\Z}1-(\frac{1-t}{p})-(\frac{t(t-1)}{p})+(\frac{-t(t-1)^2}{p})
\\& \ge p-4
\end{align*}
where we have used that $(\frac{(t-1)^2}{p})=1$ for $t\neq 1$ and the Hasse-Weil bound. Therefore the number of $t\in \Z/p\Z$ which satisfy $(\frac{t(t-1)}{p})=(\frac{1-t}{p})=-1$ is at least $\frac{p-5}{4}$ as $t=0, 1$ together contribute exactly 1 in total to the sum. For $p\ge31$, we have $\frac{p-5}{4}>6$ so for such $p$ there exists a $t$ outside of those in the set $\{0,1,\frac{1}{4},4,\frac{1}{9},9\}$ as required. 
\end{proof}
 Now choose any such fixed $t$ satisfying the above conditions. Consider the following adjustment of $\pi_2$:
\[\pi_2^y :=  \setlength{\arraycolsep}{0pt}
  \renewcommand{\arraystretch}{1.2}
  \left\{\begin{array}{l @{\quad} l}
       (0, 0)\to (1, t) & (1, 0)\to (0, 1)
       \\(0, 1)\to (0, \frac{1}{y}) & (1, 1)\to (0,0)
       \\(0, y)\to (1, 0) & (1, x)\to (1, \frac{t}{x}), x\notin\left\{0, 1\right\}
       \\ (0, x)\to (0, \frac{1}{x}), x\notin\left\{0, 1,y\right\}.
  \end{array}\right.\]
 We claim that there exists a $y$ for which $\pi_2^y$, which is $\pi_2$ with the values of $(0, y)$ and $(0, 1)$ exchanged, is AP-Destroying permutation for some choice of $y$. In particular, we claim the following.
 \begin{lemma}
Suppose that 
\[y\not\in\left\{0,1,-1,2, \frac{1}{2},\frac{1}{3}, 4, \frac{4}{t}, \frac{1}{2t+1}\right\},\]
\[(\frac{1-ty}{p})=(\frac{1-y}{p})=(\frac{(4t-1)^2y^2-2(4t+1)y+1}{p})=-1,\] and that \[(\frac{1-9y}{p})=1.\]
Then $\pi_2^y$ is AP.
 \end{lemma}
 \begin{proof}
 Note that the only difference between $\pi_2$ and $\pi_2^y$ is the exchange of $(0,1)$ and $(0,y)$. Note that this transposition destroys the APs $\{(1,0),(0,0),(1,0)\}$ and $\{(0,0),(1,0),(0,0)\}$, and it suffices to demonstrate that we created no new APs. Due to Lemma \ref{Lemma2pPartA} these APs must contain $(0,1)$ or $(0,y)$. We have four cases.
 \begin{enumerate}[\text{Case }1.]
 \item $T$ contains $(0,y)$ and $T_p=(y, y+r, y+2r)$. We take two cases based on the possibilities for $T_2$.
 \begin{enumerate}[\text{Case 1.}a.]
 \item $T_2=(0, 0, 0)$. First note that $r\nequiv 0$, as otherwise $T$ is a trivial AP. Then if $y+2r\nequiv 0$ it follows that $U_2=(1,\cdot,0)$ which is never an AP. If $y+2r\equiv 0$ then $T_p=(y,\frac{y}{2},0).$ Since $y\neq 2$, $U_p=(0,\frac{2}{y},t)$ but $y\nequiv \frac{4}{t}$ so this is not an AP.
 \item $T_2=(0, 1, 0)$. If $r\equiv 0$, then $U_p=(0, \frac{t}{y}, 0)$, which is never an AP since $t\nequiv 0$. Otherwise, if $y+2r\nequiv0$ then $U_2=(1,\cdot,0)$ which is not an AP. If $y+2r\equiv 0$ then $T_p=(y,\frac{y}{2},0).$ Since $y\nequiv 2$ then $U_p=(0,\frac{2t}{y},t)$ but $y\nequiv 4$ so this is not an AP.
 \end{enumerate}
  \item $T$ contains $(0,y)$ and $T_p=(y-r, y, y+r)$. We take two cases based on the possibilities for $T_2$.
 \begin{enumerate}[\text{Case 2.}a.]
 \item $T_2=(0, 0, 0)$. First note that $r\nequiv 0$, as otherwise $T$ is a trivial AP. If $\left\{y+r, y-r\right\}\cap \left\{0, 1\right\}=\emptyset$, then $U_p=(\frac{1}{y-r}, 0, \frac{1}{y+r})$, which is not an AP as $y\nequiv 0$. By symmetry, it suffices to check the cases $y-r\equiv 0, 1$. If $y-r\equiv0$ then $y+r\equiv2y\nequiv 1$ as $y\nequiv \frac{1}{2}$. Therefore we have $U_2=(1, 1, 0)$, which is not an AP. In the case $y-r\equiv1$, we have $y+r\equiv2y-1\notin\left\{1, y\right\}$. Now $2y-1\nequiv 0$, as $y\nequiv \frac{1}{2}$. Therefore, $U_p=(\frac{1}{y},0, \frac{1}{2y-1})$, which is not an AP as $y\nequiv \frac{1}{3}$. 
 \item $T_2=(1, 0, 1)$. If $r\equiv 0$, $U_p=(\frac{t}{y},0,\frac{t}{y})$ which is never an AP as $t\neq 0$ and thus $r\nequiv 0$ suffices. If $\left\{y+r, y-r\right\}\cap \left\{0, 1\right\}=\emptyset$, then $U_p=(\frac{t}{y-r}, 0, \frac{t}{y+r})$, which is not an AP as $yt\nequiv 0$. If $y-r\equiv 0$, then $y+r\nequiv 1$ as $y\nequiv \frac{1}{2}$. Furthermore since $y+r\nequiv 0$ it follows that $U_2=(0,1,1)$ which is not an AP. If $y-r\equiv 1$, then $y+r\nequiv 0$ as $y\nequiv\frac{1}{2}$. Since $y+r\nequiv 1$ it follows that $U_2=(0,1,1)$ which is never an AP. 
 \end{enumerate}
 Note that in the following two cases, we may assume that $T$ does not contain $(0, y)$ as these have been handled.
  \item $T$ contains $(0,1)$ and $T_p=(1, 1+r, 1+2r)$. We take two cases based on the possibilities for $T_2$.
 \begin{enumerate}[\text{Case 3.}a.]
 \item $T_2=(0, 0, 0)$. First note that $r\nequiv 0$, as otherwise $T$ is a trivial AP. If $1+2r\in \left\{0,y\right\}$, then $U_2=(1,\cdot,0)$ is not an AP. If $1+r\equiv 0$, then $U_p=(\frac{1}{y}, t, -1)$, which is impossible as $y\nequiv \frac{1}{2t+1}$. Similarly, $1+r\equiv y$ yields $U_p=(\frac{1}{y}, 0, \frac{1}{2y-1})$, which is not an AP since $y\nequiv \frac{1}{3}$. Therefore it suffices to study the general case where $U_p=(\frac{1}{y}, \frac{1}{1+r}, \frac{1}{1+2r})$. The condition for this being an AP is a quadratic in $r$ and has discriminant $(9y-1)(y-1)$. This is not a perfect square as $(\frac{1-9y}{p})=1$ and $(\frac{1-y}{p})=-1$ by assumption.
 \item $T_2=(0, 1, 0)$. If $r\equiv 0$, then $U_p=(\frac{1}{y}, 0, \frac{1}{y})$, which is not an AP. If $1+2r\in\left\{0, y\right\}$, then $U_2=(0,\cdot,1)$ is not an AP. If $1+r\equiv 0$, then since $y\nequiv -1$, we have $U_p=(\frac{1}{y}, 1, -1)$, which is not an AP as $y\nequiv \frac{1}{3}$. Finally, in the general case we have $U_p=(\frac{1}{y}, \frac{t}{1+r}, \frac{1}{1+2r})$. The condition for this sequence being an AP is a quadratic in $r$, and its discriminant is $(1-4t)^2y^2-2(4t+1)y+1$. However this is not a perfect square by assumption.
 \end{enumerate}
  \item $T$ contains $(0,1)$ and $T_p=(1-r, 1, 1+r)$. We take two cases based on the possibilities for $T_2$.
 \begin{enumerate}[\text{Case 4.}a.]
 \item $T_2=(0, 0, 0)$. First note that $r\nequiv 0$, as otherwise $T$ is a trivial AP. If $1-r\equiv 0$, then $2\equiv 1+r\notequiv y$. 
 It follows that $U_2=(1,0,0)$, which is not an AP. Since we can assume $1-r\nequiv y$ due to previous cases and we can reverse the AP as necessary, it suffices to consider $\{1-r,1+r\}\cap \{0,1,y\}=\emptyset$. In the remaining cases it follows that $U_p=(\frac{1}{1-r}, \frac{1}{y}, \frac{1}{1+r})$, which implies $y\equiv 1-r^2$. But this is impossible since $(\frac{1-y}{p})=-1$.
 \item $T_2=(1, 0, 1)$. If $r\equiv 0$, then $U_p=(0, \frac{1}{y}, 0)$, which is never an AP. If $1-r\equiv 0$, then $1+r\equiv 2$, so $U_2=(0,0,1)$ which is not an AP. Finally, in the general case $U_p=(\frac{t}{1-r}, \frac{1}{y}, \frac{t}{1+r})$. The condition for this being an AP is equivalent to $r^2\equiv1-ty$, which is impossible as $(\frac{1-ty}{p})=-1$.
 \end{enumerate}
 \end{enumerate}
 This exhausts all possible cases, so the proof is complete. 
 \end{proof}
Having shown this, we finally proceed to showing the existence of $y$ which satisfies the hypotheses of Lemma $12$. 
\begin{lemma}\label{Lemma2pPart3} 
For $p>500$ and a fixed $t$ which satisfies the hypotheses of Lemma 10, there exists a $y$ which satisfies the hypotheses of Lemma 12. 
\end{lemma}
\begin{proof}
Let $f_1=y^2(4t-1)^2-2(4t+1)y+1$. We consider
\[\sum_{y\in \Z/p\Z}(1-(\frac{1-ty}{p}))(1-(\frac{1-y}{p}))(1-(\frac{f_1}{p}))(1+(\frac{(1-9y)}{p})).\]
Expanding this product yields the $p$ plus $15$ terms of the form $\sum_{y=0}^{p-1}\pm(\frac{\pm g(y)}{p})$ where $g(y)$ is the product of some terms in the set
\[\{1-y, 9y-1, 1-ty, f_1\}.\]
We claim that none of the $g(y)$ which arise are perfect squares. To see this we instead prove the stronger claim that no two terms share a root and thus it suffices to show that the discriminant of the product is nonzero. In particular \[\Delta((y-1)(9y-1)(ty-1)f_1)=2^{28} t^3 (t-9)^2 (t-4)^2 (t-1)^8 (9 t-1)^2\] and all roots of the discriminant are in the set of excluded $t$. Hence each of the $15$ sums is at most $4\sqrt{p}+1$ in absolute value using the Hasse-Weil bound, so the entire sum is at least $p-60\sqrt{p}-15$. When $p>10000$, we have $\frac{p-60\sqrt{p}-15}{16}\ge \frac{40\sqrt{p}-15}{16}>9$. So, more than $9$ values of $y$ contribute a nonzero term to the above sum, which means that some $y$ outside of the required exceptional set satisfies
\[(\frac{1-ty}{p})=(\frac{1-y}{p})=(\frac{(1-4t)^2y^2-2(4t+1)y+1}{p})=-1\] and \[(\frac{1-9y}{p})=1\]
as required. Hence there exists an AP-Destroying permutation for $n=2p, p>10000$. In the cases $500<p<10000$, the existence of $y$ is verified in LegrendeSymbol2p.java. 
\end{proof}
\section{AP-Destroying Permutations for $\Z/3p\Z$}
For each constant $t\in \Z/p\Z, t\notin\{0, 1\}$, we can define the following permutation:
\[\pi_3 :=  \setlength{\arraycolsep}{0pt}
\renewcommand{\arraystretch}{1.2}
\left\{\begin{array}{l @{\quad} l @{\quad}l}
       (0,  0)\to (1, 0) & (0, 1)\to (1, 1) & (0, x)\to (0, \frac{1}{x}), x\notin \{0, 1\}
       \\ (1,  0)\to (2, t) & (1, 1)\to (2, 0) & (1, x)\to (1, \frac{1}{x}), x\notin \{0, 1\}
       \\ (2,  0)\to (0, 1) & (2, 1)\to (0, 0) & (2, x)\to (2, \frac{t}{x}), x\notin \{0, 1\}
\end{array}\right.\]
\begin{lemma}
Suppose that $t\in \Z/p\Z$ such that
\[t\notin \left\{-1, 0, 1, \frac{1}{2}, 2, 9\right\}\]
and 
\[(\frac{t(t-1)}{p})=(\frac{(t-1)(t-9)}{p})=-1.\]
Then $\pi_3$ is AP. 
\end{lemma}
\begin{proof}
Suppose for sake of contradiction that some arithmetic progression $T$ is preserved, and let $U$ be its image. Denote by $T_3, T_p, U_3, U_p$ the projections of $T$ and $U$ modulo $3$ and $p$ respectively. We take three cases:
\begin{enumerate}[\text{Case }1.]
\item Three elements of $T_p$ are in $\{0, 1\}$. Then since $T_p$ is an AP and $p>2$, this implies $T_p=(0, 0, 0)$ or $T_p=(1, 1, 1)$. In the former case, $U_p$ is a permutation of $(0, 1, t)$, which is not an AP as $t\notin\{-1, \frac{1}{2}, 2\}$. In the latter case, $U_p$ is a permutation of $(1, 0, 0)$, which is not an AP. Hence case $1$ is impossible.
\item One or two elements of $T_p$ are in $\{0, 1\}$. Consider the triple $T_3'$ obtained by incrementing each of the three elements in $T_3$. Note that $\pi_3$ increments the mod $3$ value of its input if that input is $0$ or $1$ mod $p$, and otherwise the mod $3$ value stays the same. It follows that if one element of $T_p$ is in $\{0, 1\}$, then $U_3$ differs from $T_3$ in exactly one element, and if two elements of $T_p$ are in $\{0, 1\}$, then $U_3$ differs from $T_3'$ in exactly one element. In both cases, $U_3$ cannot be an AP.
\item None of the elements of $T_p$ are in $\{0, 1\}$. Let $T_p=(a-r, a, a+r)$. Then we take four cases based on the possible values of $T_3$.
\begin{enumerate}[\text{Case 3.}a.]
\item $T_3=(0, 0, 0)$ or $(1, 1, 1)$. Then $U_p=(\frac{1}{a-r}, \frac{1}{a}, \frac{1}{a+r})$. It follows that $\frac{1}{a-r}+\frac{1}{a+r}=\frac{2}{a}$, so $r\equiv 0$, which is impossible.
\item $T_3=(2, 2, 2)$. Then $U_p=(\frac{t}{a-r}, \frac{t}{a}, \frac{t}{a+r})$. Since $t\nequiv 0$, this reduces to the previous case.
\item $T_3=(0, 1, 2), (1, 0, 2), (2, 0, 1),$ or $(2, 1, 0)$. By symmetry, we may suppose $T_3$ is of one of the first two triplets. Then $U_p=(\frac{1}{a-r}, \frac{1}{a}, \frac{t}{a+r})$. Solving the AP condition as a quadratic in $r$, we obtain a discriminant $(t-1)(t-9)$. This, however, is not a perfect square mod $p$ by assumption. 
\item $T_3=(0, 2, 1)$ or $(1, 2, 0)$. Then $U_p=(\frac{1}{a-r}, \frac{t}{a}, \frac{1}{a+r})$. Solving the AP condition as a quadratic in $r$, we obtain a discriminant $16t(t-1)$. This, however, is not a perfect square mod $p$ by assumption. 
\end{enumerate}
\end{enumerate}
\end{proof}
Now we prove that for $p\ge 31$, some prime satisfying the conditions of Lemma 8 exists. 
\begin{lemma}
For $p\ge 31$, there exists a $t\in \Z/p\Z$ such that
\[t\notin \{-1, 0, 1, \frac{1}{2}, 2, 9\}\]
and 
\[(\frac{t(t-1)}{p})=(\frac{(t-1)(t-9)}{p})=-1.\]
\end{lemma}
\begin{proof}
We may calculate
\[
\sum_{t\in\Z/p\Z}(1-(\frac{(t-1)(t-9)}{p}))(1-(\frac{t(t-1)}{p}))\]\[= p+\sum_{t\in\Z/p\Z}((\frac{t(t-9)(t-1)^2}{p})-(\frac{t(t-1)}{p})-(\frac{(t-1)(t-9)}{p}))\ge p-4
\]
where we have used the Hasse-Weil Bound and that $(\frac{(t-1)^2}{p})=1$ for $t\nequiv 1$.
It follows that the number of solutions to $(\frac{t(t-1)}{p})=(\frac{(t-1)(t-9)}{p})=-1$ over $t\in\Z/p\Z$ is at least $\frac{p-5}{4}>6$, so that there is in particular some $t$ outside of the exceptional set satisfying these conditions. For this value of $t$, $\pi_3$ is an AP-Destroying permutation, as desired.
\end{proof}
\section{AP-Destroying Permutations for $\Z/5p\Z$}
For each constant $t\in \Z/p\Z, t\notin\{-1, 0, 1\}$, we can define the following permutation:\[\pi_5 :=  \setlength{\arraycolsep}{0pt}
\renewcommand{\arraystretch}{1.2}
\left\{\begin{array}{l @{\quad} l @{\quad}l}
      (0, 0)\to (3, 1) & (0, 1)\to (3, 0) & (0, x)\to (0, \frac{t}{x}), x\notin\{0, 1\}
      \\ (1, 0)\to (2, 0) & (1, 1)\to (2, t) & (1, x)\to (1, \frac{t+1}{x}), x\notin\{0, 1\}
      \\ (2, 0)\to (1, t+1) & (2, 1)\to (1, 0) & (2, x)\to (2, \frac{t}{x}), x\notin\{0, 1\}
      \\ (3, 0)\to (4, 1) & (3, 1)\to (4, 0) & (3, x)\to (3, \frac{1}{x}), x\notin\{0, 1\}
      \\ (4, 0)\to (0, t) & (4, 1)\to (0, 0) & (4, x)\to (4, \frac{1}{x}), x\notin\{0, 1\}.
\end{array}\right.\]

We first note two properties of the permutation $\sigma: \Z/5\Z\to \Z/5\Z$ defined by $\sigma(0)=3, \sigma(1)=2, \sigma(2)=1, \sigma(3)=4, \sigma(4)=0$. The first is that $\sigma(i)\neq i$ for each $i$, so that in particular no AP with exactly two elements in the rightmost column can be preserved. Also, the only APs preserved by $\sigma$ are $(3, 1, 4), (0, 1, 2)$, and their reverses. In particular, every AP preserved by $\sigma$ contains $1$. 
\begin{lemma}
Suppose that $t\in \Z/p\Z$ such that
\[t\notin \{-3,-2, -1, 0, 1, 2, 3, 4, -\frac{3}{2}, -\frac{4}{3}, -\frac{3}{4}, -\frac{2}{3}, -\frac{1}{2},-\frac{1}{3}, \frac{1}{4}, \frac{1}{3}, \frac{1}{2},\frac{2}{3}, \frac{3}{4}, \frac{3}{2}\}\]
and 
\[(\frac{9t-16}{p})=(\frac{9-16t}{p})=(\frac{t+1}{p})=(\frac{(t-1)(t-9)}{p})=(\frac{(t-1)(9t-1)}{p})=-1, (\frac{t}{p})=1.\]
Then $\pi_5$ is AP. 
\end{lemma}
\begin{proof}
Suppose for sake of contradiction that some arithmetic progression $T$ is preserved, and let $U$ be its image. Denote by $T_5, T_p, U_5, U_p$ the projections of $T$ and $U$ modulo $5$ and $p$ respectively. We take four cases:
\begin{enumerate}[\text{Case }1.]
\item Three of the elements of $T_p$ are in $\{0, 1\}$. Then since $T_p$ is an $AP$, we must have either $T_p=(0, 0, 0)$ or $T_p=(1, 1, 1)$. In the first case, $U_p$ is an AP formed with elements in $\{0, 1, t, t+1\}$ not all equal. But this is impossible as $t\notin\{-2, -1, 0, 1, 2, \frac{1}{2}, -\frac{1}{2}\}$. The second case is also impossible since the only APs preserved by $\sigma$ contain $1$.
\item Two of the elements of $T_p$ are in $\{0, 1\}$. Then there are three possible values of $T_p$ up to symmetry. 
\begin{enumerate}[\text{Case 2.}a.]
\item $T_p=(0, \frac{1}{2}, 1)$. Then the first element of $U_p$ is in $\{0, 1, t, t+1\}$, the middle element is in $\{2, 2t, 2(t+1)\}$, and the last element is in $\{0, t\}$. Checking the $24$ potential combinations, there are no APs for $t$ outside of the set
\[\{-2, -1, 0, 1, 2, 3, 4, -\frac{3}{2}, -\frac{4}{3}, -\frac{3}{4}, \frac{1}{4}, \frac{1}{3}, \frac{1}{2}, \frac{3}{2}\}\]
\item $T_p=(-1, 0, 1)$. Then the first element of $U_p$ is in $\{-t-1, -t, -1\}$, the middle element is in $\{1, 0, t, t+1\}$, and the last element is in $\{0, t\}$. One of the $24$ possible combinations is $(-t, 0, t)$. However, this can only be the case if $T_5=(0, 1, 1)$ or $(2, 1, 1)$, neither of which are APs. Checking the remaining $23$ potential combinations, there are no APs for $t$ outside of the set
\[\{-3, -2, -1, 0, 1, 3, -\frac{3}{2}, -\frac{2}{3}, -\frac{1}{2}, -\frac{1}{3}\} \]
\item $T_p=(0, 1, 2)$. Then the first element of $U_p$ is in $\{0, 1, t, t+1\}$, the second element is in $\{0, t\}$, and the third is in $\{\frac{1}{2}, \frac{t}{2}, \frac{t+1}{2}\}$. Checking the $24$ potential combinations, there are no APs for $t$ outside of the set
\[\{-3, -2, -1, 0, 1, 2, 3, -\frac{3}{2}, -\frac{2}{3}, -\frac{1}{2}, -\frac{1}{3}, \frac{1}{4}, \frac{1}{3}, \frac{1}{2}, \frac{2}{3}, \frac{3}{4}\}\]
\end{enumerate}
\item One of the elements of $T_p$ is in $\{0, 1\}$. Then since $\sigma(i)\neq i$ for $0\le i\le 4$, it follows that $U_5$ cannot be an AP.
\item None of the elements of $T_p$ are in $\{0, 1\}$. Let $T_p=(a-r, a, a+r)$. If all coordinates of $T_5$ are equal, then since $t\notin\{0, -1\}$ we would have $\frac{1}{a-r}+\frac{1}{a+r}\equiv \frac{2}{a}$. But this implies $r=0$, which is impossible. Then there are seven remaining cases based on the possible values of $T_5$, up to reverses. 
\begin{enumerate}[\text{Case 4.}a.]
\item $T_5=(0, 1, 2)$. Then $U_p=(\frac{t}{a-r}, \frac{t+1}{a}, \frac{t}{a+r})$. The condition that this is an AP is a quadratic in $r$ with discriminant $t+1$, which isn't a perfect square by assumption.
\item $T_5=(0, 2, 4)$ or $(2, 0, 3)$. Then $U_p=(\frac{t}{a-r}, \frac{t}{a}, \frac{1}{a+r})$. The condition that this is an AP is a quadratic in $r$ with discriminant $(9t-1)(t-1)$, which isn't a perfect square by assumption.
\item $T_5=(0, 3, 1)$ or $(2, 4, 1)$. Then $U_p=(\frac{t}{a-r}, \frac{1}{a}, \frac{t+1}{a+r})$. The condition that this is an AP is a quadratic in $r$ with discriminant $9-16t$, which isn't a perfect square by assumption.
\item $T_5=(0, 4, 3)$ or $(2, 3, 4)$.  Then $U_p=(\frac{t}{a-r}, \frac{1}{a}, \frac{1}{a+r})$. The condition that this is an AP is a quadratic in $r$ with discriminant $(t-1)(t-9)$, which isn't a perfect square by assumption.
\item $T_5=(3, 1, 4)$. Then $U_p=(\frac{1}{a-r}, \frac{t+1}{a}, \frac{1}{a+r})$. The condition that this is an AP is a quadratic in $r$ with discriminant $t(t+1)$, which isn't a perfect square by assumption.
\item $T_5=(1, 0, 4)$ or $(1, 2, 3)$. Then $U_p=(\frac{t+1}{a-r}, \frac{t}{a}, \frac{1}{a+r})$. The condition that this is an AP is a quadratic in $r$ with discriminant $t(9t-16)$, which isn't a perfect square by assumption.
\end{enumerate}
\end{enumerate}

\end{proof}
\begin{lemma}
For $p>500$ there exists a $t$ such that
\[t\notin \{-3,-2, -1, 0, 1, 2, 3, 4, -\frac{3}{2}, -\frac{4}{3}, -\frac{3}{4}, -\frac{2}{3}, -\frac{1}{2},-\frac{1}{3},\frac{1}{3}, \frac{1}{4}, \frac{1}{3}, \frac{1}{2},\frac{2}{3}, \frac{3}{4}, \frac{3}{2}\}\]
and 
\[(\frac{9t-16}{p})=(\frac{9-16t}{p})=(\frac{t+1}{p})=(\frac{(t-1)(t-9)}{p})=(\frac{(t-1)(9t-1)}{p})=-1, (\frac{t}{p})=1.\]
\end{lemma}
\begin{proof}
We consider
\[\sum_{t\in \Z/p\Z}(1-(\frac{9t-16}{p}))(1-(\frac{9-16t}{p}))(1-(\frac{t+1}{p}))\]\[(1-(\frac{(t-1)(t-9)}{p}))(1-(\frac{(9t-1)(t-1)}{p}))(1+(\frac{t}{p})).\]
Expanding this product yields the $p$ plus $63$ terms of the form $\sum_{t\in\Z/p\Z}\pm(\frac{\pm f(t)}{p})$ where $f(t)$ is the product of some terms in the set
\[\{9t-16,9-t,1+t,(t-1)(t-9),(t-1)(9t-1),t\}.\]
We claim that none of the $f(y)$ which arise are perfect squares. To see this it suffices note that the roots $\left\{\frac{16}{9},\frac{9}{16},0,1,-1,9,\frac{1}{9}\right\}$ are all distinct for $p>500$ and no terms involving both $(t-1)(t-9)$ and $(9t-1)(t-1)$ give perfect squares. Upon expanding it can be verified that we get $4$ terms of degree $1$, $9$ terms of degree $2$, $16$ terms of degree $3$, $19$ terms of degree $4$, $12$ terms of degree $5$, and $3$ terms of degree $6$. Using the Hasse Weil bound it follows that 
\[\sum_{y\in \Z/p\Z}(1-(\frac{9t-16}{p}))(1-(\frac{9-16t}{p}))(1-(\frac{t+1}{p}))\]\[(1-(\frac{(t-1)(t-9)}{p}))(1-(\frac{(9t-1)(t-1)}{p}))(1+(\frac{t}{p}))\]
\[\ge p-13-35(2\sqrt{p}+1)-15(4\sqrt{p}+1)\] while the sum over the excluded $t$ is at most $22(64)=1408$ and the sum over the roots not in the excluded set is at most $4(64)=256$. It follows that for $p>21000$ that the sum in question is greater than $1408+256$ so such a $t$ exists and for $500<p\le 21000$ the existence of such $t$ is verified in LegrendeSymbol5p.java. 
\end{proof}
\section{AP-Destroying Permutations for $\Z/7p\Z$}
For each constant $t\in \Z/p\Z, t\notin\{0, 1\}$, we can define the following permutation:\[\pi_7 :=  \setlength{\arraycolsep}{0pt}
\renewcommand{\arraystretch}{1.2}
\left\{\begin{array}{l @{\quad} l @{\quad}l}
      (0, 0)\to (0, 1) & (0, 1)\to (0, 0) & (0, x)\to (6, \frac{t}{x}), x\notin\{0, 1\}
      \\ (1, 0)\to (1, 1) & (1, 1)\to (1, 0) & (1, x)\to (0, \frac{1}{x}), x\notin\{0, 1\}
      \\ (2, 0)\to (2, 0) & (2, 1)\to (2, t) & (2, x)\to (4, \frac{1}{x}), x\notin\{0, 1\}
      \\ (3, 0)\to (3, 1) & (3, 1)\to (3, 0) & (3, x)\to (2, \frac{t}{x}), x\notin\{0, 1\}
      \\ (4, 0)\to (5, 1) & (4, 1)\to (5, 0) & (4, x)\to (3, \frac{1}{x}), x\notin\{0, 1\}
      \\ (5, 0)\to (6, t) & (5, 1)\to (6, 0) & (5, x)\to (5, \frac{1}{x}), x\notin\{0, 1\}
      \\ (6, 0)\to (4, 1) & (6, 1)\to (4, 0) & (6, x)\to (1, \frac{1}{x}), x\notin\{0, 1\}.
\end{array}\right.\]

We first note several properties of the permutations $\sigma_1, \sigma_2: \Z/7\Z\to \Z/7\Z$ defined by 
\[\sigma_1(0)=0, \sigma_1(1)=1, \sigma_1(2)=2, \sigma_1(3)=3, \sigma_1(4)=5, \sigma_1(5)=6, \sigma_1(6)=4\]
\[\sigma_2(0)=6, \sigma_2(1)=0, \sigma_2(2)=4, \sigma_2(3)=2, \sigma_2(4)=3, \sigma_2(5)=5, \sigma_2(6)=1\]
The first is that both $\sigma_1$ and $\sigma_2$ are almost AP-Destroying; that is, they each only preserve two APs up to reversals. Namely, $\sigma_1$ preserves $(0, 1, 2)$ and $(1, 2, 3)$ while $\sigma_2$ preserves $(1, 4, 0)$ and $(4, 0, 3)$. Furthermore, for any AP $(a, b, c)$ mod $7$, the images $(\sigma_1(a), \sigma_2(b), \sigma_2(c))$ and $(\sigma_2(a), \sigma_1(b), \sigma_2(c))$ are not APs. 
\begin{lemma}
Suppose that $t\in \Z/p\Z$ such that
\[t\notin \{-2, -1, 0, 1, 2, 3, 4, -\frac{1}{2}, \frac{1}{4}, \frac{1}{3}, \frac{1}{2},\frac{2}{3}, \frac{3}{4}\}\]
and 
\[(\frac{(t-1)(t-9)}{p})=(\frac{(9t-1)(t-1)}{p})=-1.\]
Then $\pi_7$ is AP. 
\end{lemma}
\begin{proof}
Suppose for sake of contradiction that some arithmetic progression $T$ is preserved, and let $U$ be its image. Denote by $T_7, T_p, U_7, U_p$ the projections of $T$ and $U$ modulo $7$ and $p$ respectively. We take four cases:
\begin{enumerate}[\text{Case }1.]
\item Three of the elements of $T_p$ are in $\{0, 1\}$. Then since $T_p$ is an AP, it must be equal to $(0, 0, 0)$ or $(1, 1, 1)$. Furthermore, $\sigma_1$ only preserves the APs $(0, 1, 2)$ and $(1, 2, 3)$. In both cases, neither these nor their reverses yield APs for $U_7$. 
\item Two of the elements of $T_p$ are in $\{0, 1\}$. Then there are three cases up to symmetry according to the possible values of $T_p$. 
\begin{enumerate}[\text{Case 2.}a.]
\item $T_p=(0, \frac{1}{2}, 1)$. Consider $U_p$. The possible values of the first coordinate are $\{0, 1, t\}$, the possible values of the second coordinate are $\{2, 2t\}$, and the possible values of the third coordinate are $\{0, t\}$. Considering the $12$ possible combinations, there are no APs for
\[t\notin\{0, 2, 3, 4, \frac{1}{4}, \frac{1}{3}\}.\] 
\item  $T_p=(0, 1, 2)$. Consider $U_p$. The possible values of the first coordinate are $\{0, 1, t\}$, the possible values of the second coordinate are $\{0, t\}$, and the possible values of the third coordinate are $\{\frac{1}{2}, \frac{t}{2}\}$. Considering the $12$ possible combinations, there are no APs for
\[t\notin\{-2, 0, -\frac{1}{2}, \frac{1}{4}, \frac{1}{2}, \frac{2}{3}, \frac{3}{4}\}\]
\item $T_p=(-1, 0, 1)$. Consider $U_p$. The possible values of the first coordinate are $\{-1, -t\}$, the possible values of the second coordinate are $\{0, 1, t\}$, and the possible values of the third coordinate are $\{0, t\}$. The AP $(-t, 0, t)$ never occurs since it forces the second and third coordinates of $T_7$ to be $2$ and the first to be in $\{0, 3\}$. Considering the other $11$ possible combinations, there are no APs for
\[t\notin\{-2, -1, 0, 3, -\frac{1}{2}\}\]
\end{enumerate}
\item One of the elements of $T_p$ is in $\{0, 1\}$. Then due to the mentioned properties of $\sigma_1$ and $\sigma_2$, it follows that $U_7$ is not an AP. 
\item None of the elements of $T_p$ are in $\{0, 1\}$. Let $T_p=(a-r, a, a+r)$. Note that the $T_7$ coordinates cannot be equal, since that would force $\frac{1}{a-r}+\frac{1}{a+r}\equiv \frac{2}{a}$ or $r\equiv 0$, which is impossible. Then since $\sigma_2$ only preserves $(1, 4, 0)$ and $(4, 0, 3)$, we have two cases up to symmetry:
\begin{enumerate}[\text{Case 4.}a.]
\item $T_7=(1, 4, 0)$. Then $U_p=(\frac{1}{a-r}, \frac{1}{a}, \frac{t}{a+r})$. Solving the AP condition for $r$ yields a quadratic with discriminant $(t-1)(t-9)$, which is not a perfect square by assumption.
\item $T_7=(4, 0, 3)$. Then $U_p=(\frac{1}{a-r}, \frac{t}{a}, \frac{t}{a+r})$. Solving the AP condition for $r$ yields a quadratic with discriminant $(t-1)(9t-1)$, which is not a perfect square by assumption.
\end{enumerate}
\end{enumerate}
\end{proof}
\begin{lemma}
For $p\ge 66$ there exists a $t$ such that $t\notin \{-2, -1, 0, 1, 2, 3, 4, -\frac{1}{2}, \frac{1}{4}, \frac{1}{3}, \frac{1}{2},\frac{2}{3}, \frac{3}{4}\}$ and $(\frac{(t-9)(t-1)}{p})=(\frac{(9t-1)(t-1)}{p})=-1$.
\end{lemma}
\begin{proof}
Since there are $13$ excluded elements and $2$ additional roots of $(t-9)(t-1)$ and $(9t-1)(t-1)$, it suffices to demonstrate that \[\sum_{t\in \Z/p\Z}(1-(\frac{(t-9)(t-1)}{p}))(1-(\frac{(9t-1)(t-1)}{p}))\ge 15(4)+1.\] However note that \[\sum_{t\in \Z/p\Z}(1-(\frac{(t-9)(t-1)}{p}))(1-(\frac{(9t-1)(t-1)}{p}))\]\[=p-\sum_{t\in \Z/p\Z}(\frac{(t-9)(t-1)}{p})+ (\frac{(9t-1)(t-1)}{p})-(\frac{(t-9)(9t-1)(t-1)^2}{p})\ge p-4\] where the Hasse Weil-Bound and that $(\frac{(t-1)^2}{p})=1$ for $t\nequiv 1$ is used. Since $p-4>61$ the result follows.
\end{proof}
\section{Computational Techniques}
In the previous sections, computational techniques are often required to ensure the existence of AP-Destroying permutations. For $n=2p, 3p, 5p, 7p$ corresponding to $n\le 2500$, we  verified the existence of an AP-Destroying permutation via a descent algorithm; see DescentPermutation.java. In particular, we choose a random starting permutation, and only administer random transpositions if they decrease the total number of APs preserved. This condition can be checked in time linear in $n$ for each iteration. Empirically, the running time of this algorithm appeared to be roughly quadratic in $n$, which suggests that a random permutation descends to an AP-Destroying permutation with positive probability.

Whenever larger permutations were required, we calculated the necessary value of $t$ or $y,t$ directly, and this appears in many of the lemmas scattered throughout the proof. This was done instead of directly generating permutations due to the run time of this algorithm being empirically linear in $n$ versus quadratic for the above.
\section{Application to Finite Abelian Groups}
In the previous sections, we've classified which finite cyclic groups have AP-Destroying permutations. One particularly useful result is the following result of Hegarty which allows one to quotient out by subgroups with an AP-Destroying permutation. 
\begin{thm}\label{Product}
If there exists an AP-Destroying permutation for $H$ and $G/H$, there exists an AP-Destroying permutation for $G$. 
\end{thm}
Previously Elkies and Swaminathan \cite{Ashvin} demonstrated that all finite $p$-groups with odd order have an AP-Destroying permutation. We extend their result by classifying all finite abelian groups with odd order that have an AP destroying permutation.
\begin{thm}
Let $G$ be a finite abelian group with odd order greater than $7$. Then $G$ has an AP-Destroying permutation. 
\end{thm}
\begin{proof}
We first claim that the result holds if $\Omega(|G|)\le 2$, where $\Omega(n)$ denotes the number of prime factors of $n$ with multiplicity. Indeed, if $G$ is of the form $\Z/p\Z\times \Z/q\Z$ for primes $p\neq q$ or $\Z/p\Z$ or $\Z/p^2\Z$ for a prime $p$, then the result follows from the main theorem. Finally the case $G=(\Z/p\Z)^2$ follows from the result of Elkies and Swaminathan \cite{Ashvin}.
\\
\\ Now we consider the case $\Omega(|G|)=3$. If $G$ is in the set below, then the direct verification of the existence of an AP-Destroying permutation is in FiniteAbelian.java.
\[\{(\Z/3\Z)^2\times \Z/5\Z, (\Z/3\Z)^2\times \Z/7\Z, (\Z/5\Z)^2\times \Z/3\Z, (\Z/5\Z)^2\times \Z/7\Z, (\Z/7\Z)^2\times \Z/3\Z\] \[(\Z/7\Z)^2\times \Z/5\Z,
\Z/9\Z\times \Z/3\Z, \Z/25\Z\times \Z/5\Z, \Z/49\Z\times \Z/7\Z\}\] 
Other than the above set and cyclic groups, all other groups $G$ of odd order with $\Omega(|G|)=3$ have $\Z_p$ as a subgroup for some prime $p>11$ or are $G=(\Z/p\Z)^3$ for $p=3, 5, 7$. In the latter case the result follows from the result of Elkies and Swaminathan \cite{Ashvin} while in the former $G$ has an AP-Destroying permutation due to Theorem \ref{Product} along with the case $\Omega(|G|)\le 2$. 
\\
\\ Finally, we prove the full result using strong induction on $\Omega(|G|)$, with base cases $\Omega(|G|)\in \{2, 3\}$ established. Suppose the result holds for $2\le \Omega(|G|)\le k$, an that $\Omega(|G|)=k+1$. Then there exists some product $pq$ of two possibly equal primes $p, q$ such that there is an order $pq$ subgroup $H$ of $G$. Then $H$ and $G/H$ both have an AP-Destroying permutation by the inductive hypothesis, so $G$ does as well by Theorem 20. This completes the induction.
\end{proof}
We remark that there are infinite families of even-order abelian groups which do not have an AP-Destroying permutation. For example, the following is true, which is mentioned in Remark 4.2 in \cite{hegarty2004permutations}
\begin{proposition}
Suppose that $H$ is an abelian group with $|H|<2^k$. Then $G=(\Z/2\Z)^k\times H$ has no AP-Destroying permutation.  
\end{proposition}
\begin{proof}
Suppose otherwise, and let $\sigma:(\Z/2\Z)^k\times H\to (\Z/2\Z)^k\times H$ be such a permutation. Let $\pi_H: G\to H$ be the projection of $G$ onto the second coordinate. Then since $2^k>|H|$, there exist some $a\neq b\in (\Z/2\Z)^k$ such that $\pi_H\circ \sigma(a, 0)=\pi_H\circ \sigma(b, 0)$. But then $\{(a, 0), (b, 0), (a, 0)\}$ is an AP preserved by $\sigma$, a contradiction. So no such AP-Destroying permutation exists as required. 
\end{proof}
In particular, if the largest odd number dividing a positive integer $n$ is less than $\sqrt{n}$, then there exists a finite abelian group of order $n$ which does not have an AP-Destroying permutation.

\section{Conclusion}
In this paper, we have resolved a conjecture of Hegarty. In particular, we proved that there exists an AP-Destroying permutation for all cyclic groups of order not in the set $\{2, 3, 5, 7\}$. However, as the last section demonstrates, this result does not immediately resolve the case for all finite abelian groups, and in fact for every positive integer $k$ there is a finite abelian group whose order is a multiple of $k$ which does not have any AP-Destroying permutation. In light of this, the following question is still open.
\begin{prob}
For which even order finite abelian groups do there exist AP-Destroying permutations?
\end{prob}
\section{Acknowledgements}
This research was conducted at the University of Minnesota Duluth REU and was supported by NSF grant 1659047. The authors would like to thank Joe Gallian and Ashvin Swaminathan for suggesting the topic of AP-Destroying permutations, and Aaron Berger, Evan O'Dorney, Colin Defant and Joe Gallian for reading over the manuscript. Finally we would like to thank the anonymous referee who pinpointed numerous errors throughout the manuscript and simplified the presentation of several cases in the above analysis.

\bibliographystyle{plain}
\bibliography{bibliography.bib}
\end{document}